\documentclass[12pt,reqno,a4paper]{amsart}
\usepackage[english]{babel}
\usepackage[applemac]{inputenc}
\usepackage[T1]{fontenc}
\usepackage{palatino}
\usepackage{amsfonts}
\usepackage{amsmath}
\usepackage{amssymb}
\usepackage{amsthm}
\usepackage{calrsfs}
\usepackage{graphicx}
\usepackage{verbatim}
\usepackage{caption}
\usepackage[colorlinks = true, citecolor = black, linkcolor = black, urlcolor = black, pdfstartview=FitH]{hyperref}

\DeclareCaptionType{figurecaption}[Figure][]

\newcommand{\R}{\mathbb{R}}

\newcommand{\N}{\mathbb{N}}

\newcommand{\cB}{\mathcal{B}}

\newcommand{\cL}{\mathcal{L}}

\newcommand{\cF}{\mathcal{F}}

\newcommand{\cI}{\mathcal{I}}

\newcommand{\cP}{\mathcal{P}}

\newcommand{\por}{\operatorname{por}}

\newcommand{\upor}{\overline{\por}}

\newcommand{\spt}{\operatorname{spt}}
\newcommand{\Tan}{\operatorname{Tan}}

\renewcommand{\emptyset}{\varnothing}
\renewcommand{\epsilon}{\varepsilon}
\renewcommand{\rho}{\varrho}
\renewcommand{\phi}{\varphi}

\renewcommand{\j}{\mathtt{j}}
\newcommand{\card}{\operatorname{card}}

\theoremstyle{plain}
\newtheorem{thm}{Theorem}
\newtheorem{theorem}[thm]{Theorem}
\newtheorem{lemma}{Lemma}
\newtheorem{proposition}{Proposition}
\newtheorem*{proposition*}{Proposition}
\newtheorem{cor}{Corollary}

\newtheorem{condition}{Condition}

\theoremstyle{definition}
\newtheorem{definition}{Definition}

\newtheorem{remark}{Remark}

\newtheorem{notations}{Notations}
\newtheorem{notation}{Notation}

\numberwithin{equation}{section}

\author{Tuomas Orponen and Tuomas Sahlsten}\thanks{The research was supported by the Finnish Centre of Excellence in Analysis and Dynamics Research, and the second author also acknowledges the support from Emil Aaltonen Foundation.}
\title{Tangent measures of non-doubling measures}
\address{Department of Mathematics and Statistics, University of Helsinki, P.O.B. 68, FI-00014 Helsinki, Finland}
\email{tuomas.orponen@helsinki.fi} \email{tuomas.sahlsten@helsinki.fi}

\numberwithin{equation}{section}
\numberwithin{thm}{section}
\numberwithin{lemma}{section}
\numberwithin{proposition}{section}
\numberwithin{cor}{section}
\numberwithin{claim}{section}
\numberwithin{definition}{section}
\numberwithin{example}{section}
\numberwithin{remark}{section}
\numberwithin{notations}{section}
\numberwithin{reductions}{section}
\numberwithin{notation}{section}

\begin{document}

\sloppy

\maketitle

\begin{abstract}
We construct a non-doubling measure on the real line, all tangent measures of which are equivalent to Lebesgue measure.
\end{abstract}

\section{Introduction}\label{sec:1}

Tangent measures, see Definition \ref{deftang}, were introduced by D. Preiss in \cite{ref7} to solve an old conjecture on rectifiability and densities. Several examples show that singularity alone of a measure yields little information on its tangent measures. The constructions of Preiss \cite[Example 5.9]{ref7} and Freedman and Pitman \cite{ref3} exhibit purely singular measures on $\R$, all tangent measures of which are constant multiples of Lebesgue measure, or \emph{$1$-flat} in the language of Preiss. 
These examples, or in general measures with only $1$-flat tangent measures, are \textit{doubling} in the sense that the \textit{doubling constant}
$$D(\mu,x) := \limsup_{r \searrow 0} \frac{\mu(B(x,2r))}{\mu(B(x,r))}$$
is finite at $\mu$ almost every $x \in \R^d$. This is a consequence of the following characterisation, which is a combination of the statements of \cite[Proposition 2.2]{ref7} and \cite[Corollary 2.7]{ref7}.

\begin{thm}
\label{nondoublingcharacterisation}
Let $\mu$ be a measure on $\R^d$ and $x \in \R^d$ with $\Tan(\mu,x)\neq \emptyset$. Then
$$D(\mu,x) = \infty \quad \Longleftrightarrow \quad \sup_{\nu \in \Tan(\mu,x)} \frac{\nu(B(0,R))}{\nu(B(0,1))} = \infty\quad \text{ for every } R > 1.$$
\end{thm}

The purpose of this paper is to study the tangent measures of \textit{non-doubling measures}, that is, measures satisfying $D(\mu,x) = \infty$ at $\mu$ almost every $x \in \R^d$. Non-doubling is a particularly strong form of singularity, and, as Theorem \ref{nondoublingcharacterisation} indicates, it poses some restrictions on tangent measures -- as opposed to mere singularity, which in general has no impact on their behaviour. Motivated by this observation, we set out to study whether more would be true: do tangent measures of non-doubling measures always inherit some degree of singularity? Our investigation concluded with the following result:
\begin{thm}
\label{main}
There exists a non-doubling measure $\mu$ on $\R$ such that every tangent measure $\nu$ of $\mu$ is equivalent to Lebesgue measure.
\end{thm}

In the spirit of the examples of Freedman-Pitman and Preiss, our construction shows that not even non-doubling guarantees any form of singularity for tangent measures.

\section{An application to porosity}\label{sec:2}

Theorem \ref{main} has implications to the theory of porosity, a degree of singularity, which has attained much attention in recent years in fractal geometry.

\begin{definition}For a measure $\mu$ on $\R^d$, $x \in \R^d$ and $r,\epsilon > 0$ write
\begin{align*}
\por(\mu,x,r,\epsilon) = \sup\{\delta > 0 : \,&\mbox{there exists } y \in \R^n \mbox{ with } B(y,\delta r) \subset B(x,r)\\
& \mbox{and }\mu(B(y,\delta r)) \leq \epsilon \mu(B(x,r))\}.
\end{align*}
The \textit{upper porosity} of $\mu$ at $x$ is then defined by 
$$\upor(\mu,x) = \lim_{\epsilon \searrow 0} \limsup_{r \searrow 0} \por(\mu,x,r,\epsilon).$$
A measure $\mu$ is \textit{upper porous} if $\upor(\mu,x) > 0$ at $\mu$ almost every $x \in \R^d$.
\end{definition}

Upper porosity was introduced by J--P. Eckmann, E. J\"arvenp\"a\"a and M. J\"arvenp\"a\"a in \cite{ref1}, and further investigated by M. E. Mera, M. Mor\'an, D. Preiss and L. Zaj\'{i}\v{c}ek in articles \cite{ref5,ref6}. A wide class of examples of upper porous measures was exhibited by V. Suomala in \cite{ref8}. In \cite{ref5} Mera and Mor\'an presented a characterisation of doubling upper porous measures in terms of tangent measures:

\begin{theorem}
\label{meramoran}
A doubling measure $\mu$ on $\R^d$ is upper porous if and only if for $\mu$ almost every $x \in \R^d$ there exists $\nu \in \Tan(\mu,x)$ with $\spt \nu \neq \R^{d}$.
\end{theorem}

In 2009, Suomala asked us whether Theorem \ref{meramoran} holds without the doubling assumption. Note that the question makes sense since $\Tan(\mu,x) \neq \emptyset$ at $\mu$ almost every $x \in \R^d$ even without the doubling assumption, see \cite[Theorem 2.5]{ref7}. The non-doubling measure $\mu$ constructed in Theorem \ref{main} is upper porous by \cite[Proposition 3.3]{ref6}, yet every tangent measure $\nu$ of $\mu$ is equivalent to Lebesgue measure, so $\spt \nu = \R$. Hence $\mu$ answers Suomala's question in the negative:

\begin{cor}[to Theorem \ref{main}]
\label{counterporosity}
The measure $\mu$ in Theorem \ref{main} is upper porous, yet $\spt \nu = \R$ for every tangent measure $\nu$ of $\mu$.
\end{cor}

\section{Definitions and the construction of $\mu$} \label{sec:3}

Below, a \emph{measure} is always a locally finite Borel measure.

\begin{notations} The closed and open balls with center $x \in \R^{d}$, $d \in \N$, and radius $r > 0$ will be denoted $B(x,r)$ and $U(x,r)$, respectively. The length of an interval $I \subset \R$ is denoted by $\ell(I)$. If $f : \R^d \to \R^d$ is Borel-measurable and $\mu$ is a measure on $\R^d$, we denote by $f_{\sharp} \mu$ the \textit{push-forward} measure of $\mu$ under the map $f$, defined for $A \subset \R^d$ by $f_{\sharp} \mu(A) = \mu(f^{-1} A)$. The \textit{support} of a measure $\mu$, denoted $\spt \mu$, is the set of all $x \in \R^d$ with $\mu(B(x,r)) > 0$ for all $r > 0$. A measure $\mu$ is \textit{absolutely continuous} with respect to a measure $\tau$, denoted $\mu \ll \tau$, if $\tau(A) = 0$ implies $\mu(A) = 0$ for all Borel sets $A$. Moreover, $\mu$ are $\tau$ are \textit{equivalent} if $\mu \ll \tau \ll \mu$. We write $\cL^{d}$ for Lebesgue measure on $\R^{d}$. If $\mu$ is a measure on $\R^{d}$ and $0 \leq s \leq d$, the $s$-dimensional \emph{upper-} and \emph{lower densities} of $\mu$ at $x \in \R^{d}$ are the quantities
\begin{displaymath} \Theta^{\ast s}(\mu,x) = \limsup_{r \searrow 0} \frac{\mu(B(x,r))}{(2r)^{s}} \quad \text{and} \quad \Theta_{\ast}^{s}(\mu,x) = \liminf_{r \searrow 0} \frac{\mu(B(x,r))}{(2r)^{s}}. \end{displaymath}
\end{notations}

\begin{definition}[Tangent measures]\label{deftang}Let $\mu$ be a measure on $\R^d$ and $x \in \R^d$. A non-zero measure $\nu$ on $\R^d$ is called a \emph{tangent measure} of $\mu$ at $x$ if $\nu$ is obtained as the weak limit of the sequence $c_{i} T_{x,r_{i}\sharp}\mu$, where $(c_{i})_{i \in \N}$ and $(r_{i})_{i \in \N}$ are sequences of positive constants, $r_{i} \searrow 0$, and  $T_{x,r_i}$ is the map $T_{x,r_i}(y) = (y-x)/r_i$, $y \in \R^d$, taking $B(x,r_{i})$ to $B(0,1)$. The set of all tangent measures of $\mu$ at $x$ is denoted  $\Tan(\mu,x)$.
\end{definition}

\textit{Construction of $\mu$}. The measure $\mu$ of Theorem \ref{main} is constructed by introducing a single auxiliary function $\phi$ and then using it repeatedly as a 'rule' to distribute mass, see Figure 1. 

\begin{figure}[h]
\label{fig1}
\begin{center}
\includegraphics[scale=0.5]{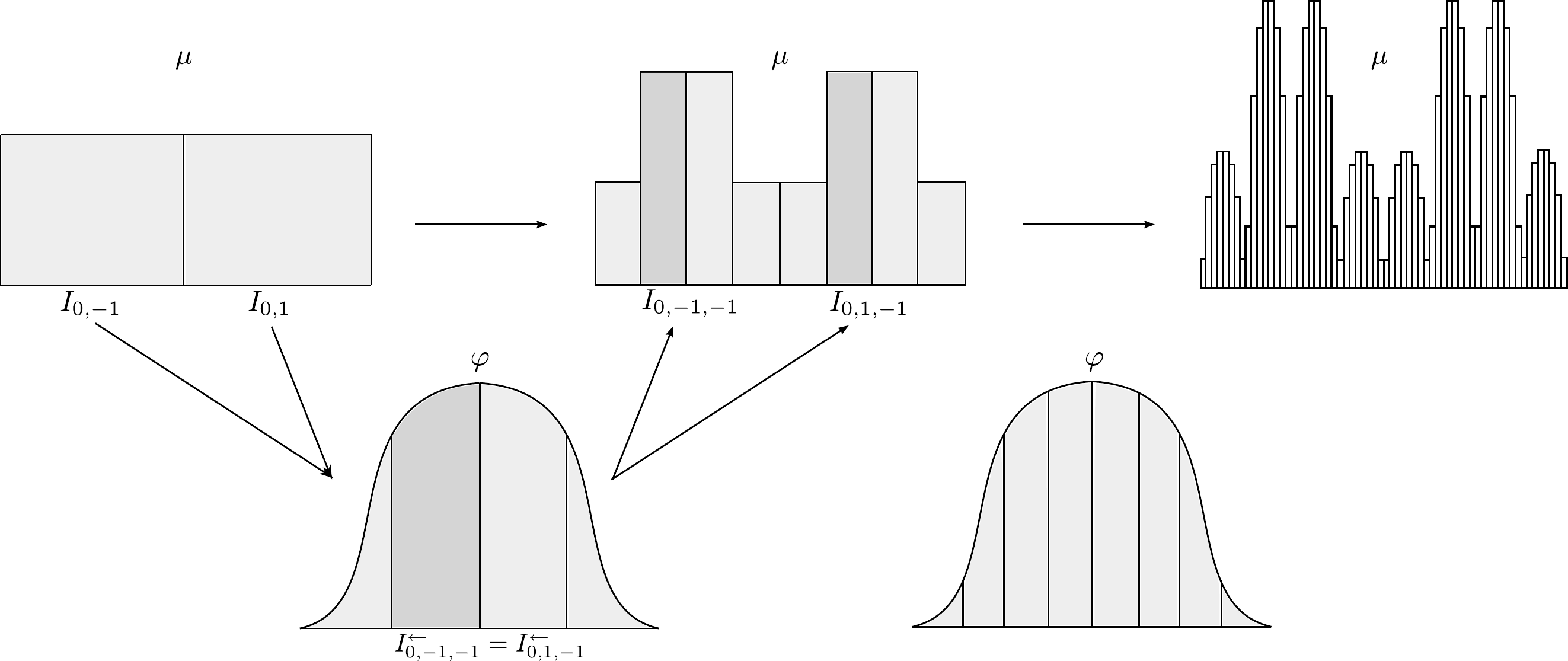}
\end{center}
\caption{The construction of the measure $\mu$: the weights of the construction intervals are determined by $\phi$-integrals over the pull-back intervals.}
\end{figure}

Let $\phi$ be the function
$$\phi(t) := c\exp \Big( \frac{1}{|t|-1} \Big), \quad t \in [-1,1],$$
where $c > 0$ is determined by the requirement $\int \phi = 1$. For $k \in \N$, we write
$$[2^k] := \{\pm i : i = 1,2,3,\dots,2^{k}\}, \quad \Sigma^0 := \{0\} \quad \text{and} \quad \Sigma^k := \Sigma^{k-1} \times [2^{k-1}].$$
Define $I_0 := [-1,1)$. Assuming that $I_\j$ is defined for some $\j \in \Sigma^k$, we divide $I_\j$ into $2^{k+1}$ half-open dyadic subintervals 
$$\cI_\j := \{I_{\j i} \subset I_\j : i \in [2^k]\}$$ 
of length $2^{-k-1} \ell(I_\j)$ enumerated from left to right. 
We call the intervals in $\cI_{\j}$, $\j \in \Sigma_{k}$, the \emph{generation $k$ construction intervals}. 
The set of all generation $k$ construction intervals will be denoted by $\cI_k$, so that $\cI_{k} = \bigcup_{\j \in \Sigma^k} \cI_\j$. 
For $A \subset [-1,1)$, we write $\phi(A) := \int_A \phi$. The $\mu$ measures of the construction intervals are defined by $\mu(I_{0}) = 1$, and
$$\mu(I_{\j i}) = \varphi(I_{\j i}^{\leftarrow})\mu(I_{\j}), \qquad \j \in \Sigma_{k},\,\, i \in [2^{k}],$$
where $I_{\j i}^{\leftarrow} \subset [-1,1)$ is the \emph{pull-back} of the interval $I_{\j i}$ defined by
$$I_{\j i}^{\leftarrow} = \begin{cases}[i2^{-k},(i+1)2^{-k}), & i < 0;\\
[(i-1)2^{-k},i2^{-k}), & i > 0.
\end{cases}$$ 
This procedure uniquely determines a probability measure on $[-1,1)$, see \cite[Proposition 1.7]{ref2}.

\section{Preliminary lemmas} \label{sec:4}

Before we can understand $\mu$, we need to establish some key properties of $\phi$: 

\begin{lemma}
\label{phiprop}
The function $\phi$ satisfies the following:
\begin{enumerate}
\item[(1)] if $0 < \tau < 1$, then for every interval $I \subset [-\tau,\tau]$, we have
$$\phi(I) \asymp_{\tau} \ell(I),$$
that is, $c^{-1} \ell(I) \leq \phi(I) \leq c \ell(I)$ for some $c > 0$ depending only on $\tau$.
\item[(2)] for every $M > 1$ and $N \in \N$ there exists $l_{N,M} > 0$ such that
$$\frac{\phi(I + \ell(I))}{\phi(I)} > M \quad \Longrightarrow \quad \frac{\phi(I + (N' + 1) \cdot \ell(I))}{\phi(I + N'\cdot\ell(I))} > M^{1/8}/2$$
for every $1 \leq N' \leq N$ and all intervals $I \subset [-1,1)$, as soon as $\ell(I) \leq l_{N,M}$;
\item[(3)] if $C > 1$, then
$$G_{C,\epsilon} := \frac{\phi([-1,-1+C\epsilon])}{\phi([-1,-1+\epsilon])} = \frac{\phi([1-C\epsilon,1])}{\phi([1-\epsilon,1])} \to +\infty, \quad \text{as } \epsilon \searrow 0.$$
\end{enumerate}
\end{lemma}

\begin{proof}
(1) Since $I \subset [-\tau,\tau]$, we have $\phi(\tau) \leq \varphi \leq \phi(0)$ on $I$. This gives
\begin{displaymath} \varphi(\tau)\ell(I) \leq \int_{I} \varphi \equiv \varphi(I) \leq \varphi(0)\ell(I). \end{displaymath}
Hence $\phi(I) \asymp_{\tau} \ell(I)$, as claimed.

(2) Fix an interval $I \subset [-1,1)$. Then, assuming that $I$ is half-open, we have $I = [x,x + \ell(I))$ for some $x \in [-1,1)$. Given $M$ and $N$, we may initially choose $l_{N,M}$ so small that
\begin{displaymath} \ell(I) \leq l_{N,M} \, \text{ and } \, I + (N + 1)\cdot \ell(I) \not\subset [-1,0] \quad \Longrightarrow \quad \frac{\varphi(I + \ell(I))}{\varphi(I)} \leq M. \end{displaymath}
Thus, for the rest of the proof, we may rest assured that $I + N' \cdot \ell(I) \subset [-1,0]$ for $0 \leq N' \leq N + 1$. Suppose that $\varphi(I + \ell(I))/\varphi(I) > M$. Since $\phi$ is increasing on $[-1,0]$, we have
\begin{displaymath} \frac{\phi(x + 2\cdot\ell(I))}{\phi(x)} \geq \frac{\phi(I + \ell(I))}{\phi(I)} > M. \end{displaymath}
Taking logarithms on both sides and invoking the definition of $\phi$ results in
\begin{displaymath} \frac{1}{x + 1} - \frac{1}{(x + 2 \cdot \ell(I)) + 1} = \frac{1}{-(x + 2\cdot\ell(I)) - 1} - \frac{1}{-x - 1} > \ln M. \end{displaymath}  
Applying the mean value theorem to $x \mapsto x^{-1}$ yields $\xi \in (x + 1,x + 2\cdot\ell(I) + 1)$ such that
\begin{displaymath} \frac{2\cdot\ell(I)}{\xi^{2}} = \frac{1}{x + 1} - \frac{1}{(x + 2\cdot\ell(I)) + 1} > \ln M, \end{displaymath}
which implies 
\begin{displaymath} 0 \leq \xi < (\ln M)^{-1/2} \cdot (2\cdot\ell(I))^{1/2}. \end{displaymath}
Now we are ready to estimate the ratio $\phi(I + (N' + 1)\cdot \ell(I))/\phi(I + N'\cdot \ell(I))$. The function $\phi$ is increasing on the intervals $I + N' \cdot \ell(I)$ and $I + (N' + 1) \cdot \ell(I)$, $1 \leq N' \leq N$, so we may estimate
\begin{align}\label{ineq} \frac{\phi(I + (N' + 1)\cdot \ell(I))}{\phi(I + N'\cdot \ell(I))} \geq \frac{\phi(x + (N'+1) \cdot \ell(I) + \ell(I)/2)}{2\phi(x + (N'+1) \cdot \ell(I))}. \end{align} 
Taking logarithms and applying the mean value theorem as above yields a point $\eta \in (x + (N'+1) \cdot \ell(I) + 1, x + (N'+1) \cdot \ell(I) + \ell(I)/2 + 1)$ such that
\begin{displaymath} \frac{\ell(I)/2}{\eta^{2}} = \ln\left(\frac{\phi(x + (N'+1) \cdot \ell(I) + \ell(I)/2)}{\phi(x + (N'+1) \cdot \ell(I))}\right). \end{displaymath}
Since $t^{1/2}$ tends to zero slower than $N \cdot t$, we may choose $l_{N,M}$ so small that $\ell(I) \leq l_{N,M}$ implies $(N'+1) \cdot \ell(I) + \ell(I)/2 \leq (\ln M)^{-1/2} \cdot (2 \cdot \ell(I))^{1/2}$ for all $1 \leq N' \leq N$. Using this and previously established bound for $\xi$ now yields 
\begin{align*} \eta & < x + (N'+1) \cdot \ell(I) + \ell(I)/2  + 1\\
& \leq \xi + (N'+1)\cdot \ell(I) + \ell(I)/2\\
& < 2 \cdot (\ln M)^{-1/2}\cdot(2\cdot\ell(I))^{1/2} \end{align*}
for intervals $I$ of length $\ell(I) \leq l_{N,M}$. The proof is finished by combining this with (\ref{ineq}) and the definition of $\eta$: 
\begin{displaymath} \frac{\phi(I + (N' + 1)\cdot \ell(I))}{\phi(I + N'\cdot \ell(I))} > \frac{1}{2}\cdot\exp\left(\frac{\ell(I)/2}{[2\cdot (\ln M)^{-1/2}) \cdot (2\cdot\ell(I))^{1/2}]^{2}}\right) = M^{1/8}/2 \end{displaymath}

(3) As $\phi$ is even, it is enough to show that $\varphi([1 - C\epsilon,1])/\varphi([1 - \epsilon,1]) \to \infty$ as $\epsilon \to 0$. Fix $\epsilon \in (0,1)$, and let $1 < D < C$. As $\varphi$ is decreasing on $[0,1]$, we have
$$\phi([1 - C\epsilon, 1]) \geq \phi([1-C\epsilon,1-D\epsilon]) \geq \phi(1 - D\epsilon)(C - D)\epsilon.$$

Combining this with the definition of $\varphi$ yields
\begin{align*} \ln G_{C,\epsilon} = \ln \frac{\phi([1-C\epsilon,1])}{\phi([1-\epsilon,1])} \geq \ln \frac{\varphi(1 - D\epsilon)(C - D)}{\varphi(1 - \epsilon)} = \ln(C - D) + \frac{D - 1}{D\epsilon}. \end{align*} 
The right hand side tends to $\infty$, so the proof is complete.
\end{proof}

The next lemma is the counterpart of Lemma \ref{phiprop} for the measure $\mu$. To state the result, we need two definitions.

\begin{definition}[Comparability] \label{compa} Let $C \geq 1$ and $0 < \lambda \leq 1$. We say that a pair $(A,B)$ of Borel sets $A,B \subset [-1,1)$ is $(C,\lambda)$-\emph{comparable}, if
$$C^{-1} \lambda \leq \frac{\mu(A)}{\mu(B)} \leq C\lambda.$$
\end{definition} 

\begin{definition}[Coverings and Packings] \label{coverpack} Let $k \in \N$. If $J$ is an interval and $J \subset I_k$ for some $I_k \in \cI_k$, then we write $\cF_{J} = \{I \in \cI_{k + 1} : I \cap J \neq \emptyset\}$ to be the minimal $\cI_{k + 1}$-\emph{cover} of $J$, and write $\cP_{J} = \{I \in \cI_{k + 1} : I \subset J\}$ to be the maximal $\cI_{k + 1}$-\emph{packing} of $J$. The reference to $k$ is suppressed from the notation, as the relevant generation will always be clear from the context.
\end{definition} 

\begin{remark}\label{dyadic}
If $\j \in \Sigma_{k}$ is fixed, $\cB$ is a collection of sets in $\cI_{\j}$ and $B = \bigcup \cB$, we extend the definition of pull-back by writing
\begin{displaymath} B^{\leftarrow} := \bigcup_{I \in \cB} I^{\leftarrow}. \end{displaymath} 
The formula
\begin{displaymath} \mu(B) = \varphi(B^{\leftarrow})\mu(I_{\j}) \end{displaymath}
then follows immediately from the definition of $\mu$.
\end{remark}

\begin{lemma}
\label{muprop}
The measure $\mu$ satisfies the following properties.
\begin{enumerate}
\item[(1)] For each $\tau \in (0,1)$ there exists a generation $k(\tau) \in \N$ and a constant $C(\tau) \geq 1$ with the following property. Let $k \geq k(\tau)$ and $I_{k} \in \cI_{k}$. If $J$ is any interval such that $J \subset I_{k}$, $\cP_J \neq \emptyset$ and
$$d(J,\partial I_k) \geq \tau\ell(I_k),$$
then $(J,I_{k})$ is $(C(\tau),\lambda)$-comparable, where $\lambda = \ell(J)/\ell(I_k)$.
\item[(2)] Let $C \geq 1$, $k \in \N$ and $I_{k} \in \cI_{k}$. Suppose that $J \subset I_k$ is an interval such that $\cP_J \neq \emptyset$ (that is, $J$ contains an interval from $\mathcal{I}_{k + 1}$), $5J \subset I_{k}$,
$$\ell(J)/\ell(I_k) < \min\left\{\frac{l_{3,6^8C^8}}{2},\frac{1}{20}\right\} \quad \text{and}\quad \mu(5J) \leq C\mu(J)$$
where $l_{3,6^8C^8}$ is the threshold from Lemma \ref{phiprop}\emph{(2)}, and $5J$ is the interval with the same midpoint as $J$ and length $\ell(5J) = 5\ell(J)$. Then all pairs of intervals in $\cF_J$ are $(D,1)$-comparable with $D = \max\{6^{25}C^{25},C(1/4)^{2}\}$ (here $C(1/4)$ is the constant from \emph{(1)} with $\tau = 1/4$).
\item[(3)] Let $C > 8$, $k \in \N$ and $I_{k} \in \cI_{k}$. Assume that $J \subset J^\star \subset I_k$ are intervals such that the packing $\cP_J \neq \emptyset$ (that is, $J$ contains an interval from $\mathcal{I}_{k + 1}$), 
$$\ell(J^\star) = C\ell(J) \quad \text{and}\quad \overline{J} \cap \partial I_k \neq \emptyset.$$ 
Then
$$\mu(J^\star)/\mu(J) \geq G_{C/8,4\lambda},$$
where $\lambda := \ell(J)/\ell(I_k)$ and $G_{C/8,4\lambda}$ is the ratio from Lemma \ref{phiprop}\emph{(3)}.
\end{enumerate}
\end{lemma}

\begin{proof}
(1) Since $d(J,\partial I_k) \geq \tau\ell(I_k)$, every interval $I \in \cF_{J}$ has
$$d(I,\partial I_k) \geq \tfrac{\tau}{2}\ell(I_k),$$
provided that $\ell(I) \leq \tau\ell(I_{k})/2$, that is, choosing $k(\tau)$ large enough and assuming $k \geq k(\tau)$. This shows that $I^{\leftarrow} \subset [-1+\tau,1-\tau]$, whence Lemma \ref{phiprop}(1) yields 
$$\varphi(I^{\leftarrow}) \asymp_{\tau} \ell(I^{\leftarrow}) \asymp 2^{-k}.$$
As $\mu(I) = \phi(I^{\leftarrow})\mu(I_k)$, this proves that $(I,I_{k})$ is $(C,2^{-k})$-comparable for some constant $C = C(\tau) \geq 1$ depending only on $\tau$. If $J = I$, we are done. If $\card\cF_{J} > 1$, the proof is reduced to this by comparing with $\mu(I_{k})$ the $\mu$ measures of $\bigcup \cF_{J}$ and $\bigcup \cP_{J} \neq \emptyset$. 

(2) First assume that $5J$ contains the midpoint of $I_k$. As $5\ell(J) < \ell(I_k)/4$ by assumption, we then have
$$d(5J,\partial I_k) \geq \frac{\ell(I_k)}{4},$$
Let $I,I' \in \cF_{J}$. Then $I,I' \subset 5J$, so the inequality above yields $d(I,\partial I_{k}) > \ell(I_{k})/4$ and $d(I',\partial I_k) > \ell(I_k)/4$. The proof of part (1) now implies that $(I,I_{k})$ and $(I',I_{k})$ are $(C(1/4),2^{-k})$-comparable, which shows that $(I,I')$ is $(C(1/4)^{2},1)$-comparable. 

Next assume that $5J$ does not contain the midpoint of $I_k$. Then $5J$ is contained in either half of $I_k$, say the right one. Let $P_0 = \bigcup \cP_{J}$, and define $P_N := P_{0}+N\ell(P_{0})$ for $N \in \{-2,-1,0,1,2\}$. 

\begin{figure}[h]
\label{fig2}
\begin{center}
\includegraphics[scale = 0.5]{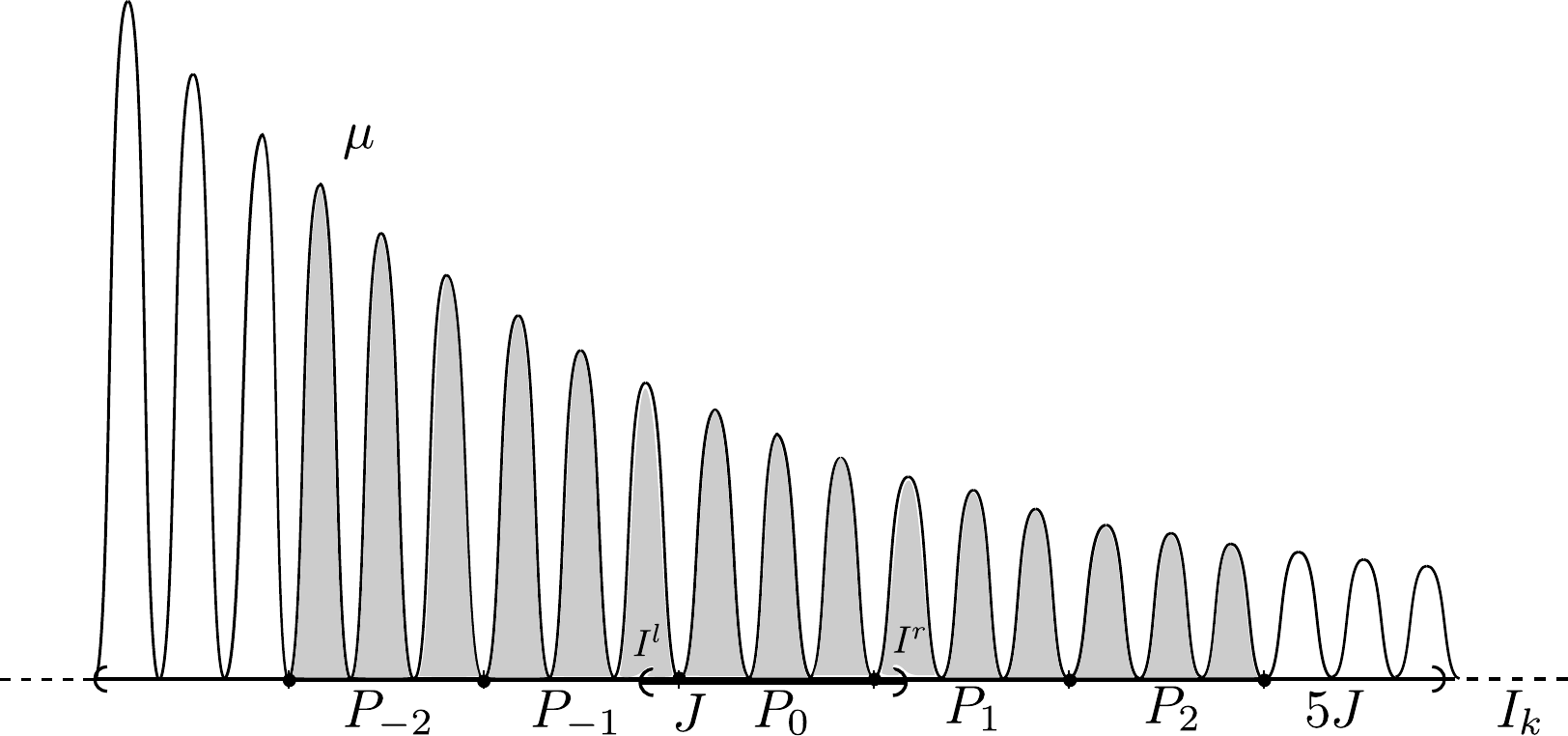}
\end{center}
\caption{The distribution of $\mu$ in the vicinity of the intervals $P_N$.}
\end{figure}

Since $\ell(P_N) \leq \ell(J)$, the intervals $P_{N}$ are contained in $5J \subset I_{k}$ for all $-2 \leq N \leq 2$, see Figure 2. As was pointed out in Remark \ref{dyadic}, we have
$$\mu(P_N) = \phi(P^{\leftarrow}_N) \mu(I_k),$$
Note that, by assumption,
\begin{equation}\label{claim}\ell(P^{\leftarrow}_N) = 2\ell(P^{\leftarrow}_{N})/\ell(I_{0}) = 2\ell(P_N)/\ell(I_k) \leq 2\ell(J)/\ell(I_k) < l_{3,6^8 C^8}.\end{equation}
Since $P_N \subset 5J$, and $5J$ is a subset of the right half of $I_k$, we see that $P^{\leftarrow}_N \subset [0,1]$ for $-2 \leq N \leq 2$. As $\phi$ is decreasing on $[0,1]$, this yields $\phi(P^{\leftarrow}_N) \geq \phi(P^{\leftarrow}_{N+1})$ for $-2 \leq N \leq 1$, see Figure 3. 

\begin{figure}[h]
\label{fig3}
\begin{center}
\includegraphics[scale = 0.5]{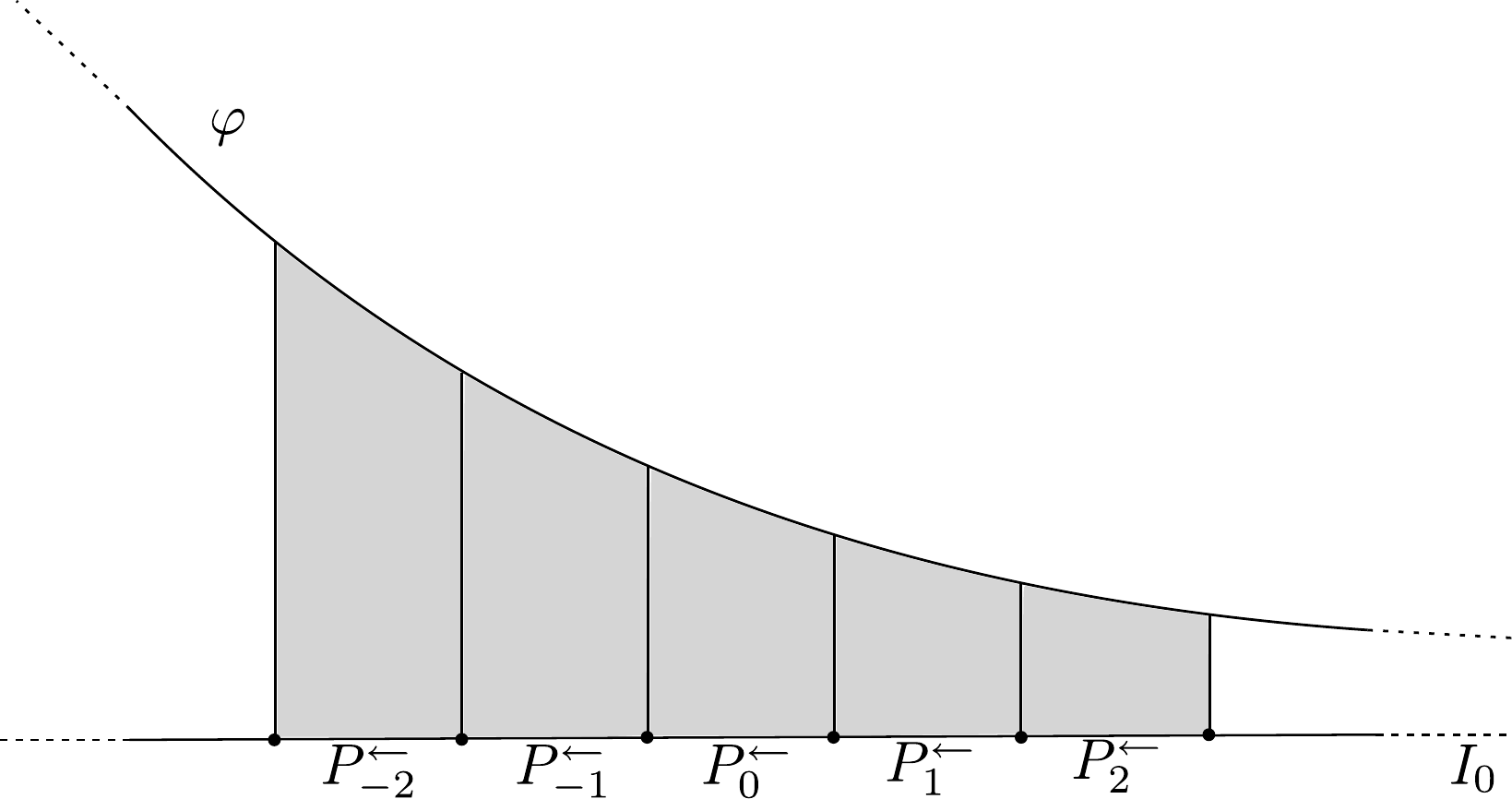}
\end{center}
\caption{The mapping $\phi$ is decreasing in the vicinity of the pull-back intervals $P_N^{\leftarrow}$.}
\end{figure}

Now
$$\mu(P_{-2}) \leq \mu(5J) \leq C\mu(J) \leq C(\mu(P_{-1}) + \mu(P_0) + \mu(P_1)),$$
so that dividing by $\mu(I_{k})$ yields
$$\phi(P^{\leftarrow}_{-2}) \leq C(\phi(P^{\leftarrow}_{-1}) + \phi(P^{\leftarrow}_0) + \phi(P^{\leftarrow}_1)) \leq 3C \phi(P^{\leftarrow}_{-1}).$$
Next we claim that $\phi(P^{\leftarrow}_{-2}) \leq 6^{25}C^{25}\phi(P^{\leftarrow}_{2})$. If not, then the estimate 
$$6^{24}C^{24} < \frac{\varphi(P^{\leftarrow}_{-1})}{\varphi(P^{\leftarrow}_{-2})} \cdot \frac{\phi(P^{\leftarrow}_{-2})}{\phi(P^{\leftarrow}_2)} = \frac{\phi(P^{\leftarrow}_{-1})}{\phi(P^{\leftarrow}_0)} \cdot \frac{\phi(P^{\leftarrow}_{0})}{\phi(P^{\leftarrow}_{1})} \cdot \frac{\phi(P^{\leftarrow}_{1})}{\phi(P^{\leftarrow}_2)}$$
shows that at least one of the three factors in the latter product must be strictly bigger than $6^8 C^8$. Recalling (\ref{claim}), Lemma \ref{phiprop}(2) then implies that $\phi(P^{\leftarrow}_{-2})/\phi(P^{\leftarrow}_{-1}) > 3C$, which is impossible. Hence $\phi(P^{\leftarrow}_{-2}) \leq 6^{25}C^{25}\phi(P^{\leftarrow}_{2})$, and consequently $\mu(P_{-2}) \leq 6^{25}C^{25}\mu(P_{2})$. Among all the intervals in $\cF_{J}$, let $I^{l}$ and $I^{r}$ be the leftmost and rightmost ones (respectively), or, in other words, the ones with the largest and the smallest $\mu$ measure. If $\card\cP_{J} = n$, we obtain
\begin{align}
\label{number} 
n\mu(I^{l}) \leq \mu(P_{-2}) \leq 6^{25}C^{25}\mu(P_{2}) \leq 6^{25}C^{25}n\mu(I^{r}), 
\end{align}
since $P_{-2}$ is the union of $n$ intervals in $\cI_{k+1}$ situated to the left of $I^{l}$ (thus having larger $\mu$ measure than $I^{l}$), and similarly $P_2$ is a union of $n$ intervals in $\cI_{k+1}$ situated to the right of $I^{r}$. Here we needed once more that $P_{N} \subset 5J$, and $5J$ lies on the right half of $I_{k}$. Dividing both sides of (\ref{number}) by $n$ finishes the proof.

(3) Let $P = \bigcup \cP_{J^\star}$ and $F = \bigcup \cF_{J}$. Since $C \geq 3$, we have $J \subset F \subset P \subset J^\star$, and
$$\ell(F) \leq 2\ell(J) = 2C^{-1}\ell(J^\star) \leq 4C^{-1}\ell(P).$$
Then, recalling that $\lambda = \ell(J)/\ell(I_{k})$, 
$$\ell(F^\leftarrow) = 2\ell(F)/\ell(I_k) \leq 4\lambda \quad \text{and} \quad \ell(P^\leftarrow) = 2\ell(P)/\ell(I_k) \geq \tfrac{C}{2} \cdot \lambda = \tfrac{C}{8} \cdot 4\lambda.$$
Since $\overline{J} \cap \partial I_k \neq \emptyset$, we have either $1 \in \overline{F^\leftarrow}$ or $-1 \in \overline{F^{\leftarrow}}$. In the former case
$$F^\leftarrow = [1-\ell(F^\leftarrow),1] \subset [1-4\lambda,1] \quad \text{and} \quad P^\leftarrow = [1-\ell(P^\leftarrow),1] \supset [1-\tfrac{C}{8} \cdot 4\lambda,1].$$
An application of Remark \ref{dyadic} and Lemma \ref{phiprop}(3) then gives
$$ \frac{\mu(J^\star)}{\mu(J)} \geq \frac{\mu(P)}{\mu(F)} = \frac{\phi(P^\leftarrow)}{\phi(F^\leftarrow)} \geq G_{C/8J,4\lambda}. $$
Since $\varphi$ is even, we reach the same conclusion if $-1 \in \overline{F^\leftarrow}$.
\end{proof}

\section{Non-doubling of $\mu$} \label{sec:5}

The non-doubling of $\mu$ can be deduced from the following

\begin{lemma}
\label{nondoubling}
For $\mu$ almost every $x \in \R$ there exists an increasing sequence $(k_{i})_{i \in \N}$ of integers such that $2^{-i}\ell(I) \leq d(x,\partial I) \leq 2^{-i + 1}\ell(I)$ whenever $x \in I \in \cI_{k_{i}}$.
\end{lemma}

Indeed, assume that $x$ is one of the points satisfying the description above and choose the sequence $(k_{i})_{i \in \N}$ accordingly. If $x \in I \in \mathcal{I}_{k_{i}}$, set $r_{i} = d(x,\partial I)$, $J_{i} = B(x,r_{i})$ and write $J^\star_{i} = B(x,17r_{i}) \cap I$. Since $r_{i} \geq 2^{-i}\ell(I) \geq 2^{-k_{i}}\ell(I)$, the interval $J_{i}$ contains at least one interval in $\cI_{k_{i} + 1}$, and $\ell(J^\star_{i}) = 9\ell(J_{i})$. Since $J_{i} \subset I$, Lemma \ref{muprop}(3) now yields
\begin{displaymath} \frac{\mu(B(x,17r_{i}))}{\mu(B(x,r_{i}))} \geq \frac{\mu(J^\star_{i})}{\mu(J_{i})} \geq G_{9/8,4\lambda(i)}, \end{displaymath}
where $\lambda(i) = \ell(J_{i})/\ell(I) \leq 2^{-i + 2}$. Since $\lambda(i) \to 0$, as $i \to \infty$, we obtain
$$\limsup_{r \searrow 0} \frac{\mu(B(x,17r))}{\mu(B(x,r))} = \infty,$$
by Lemma \ref{phiprop}(3). This is equivalent to $D(\mu,x) = \infty$.

\begin{proof}[Proof of Lemma \ref{nondoubling}] For $i,k \in \N$, consider the sets
\begin{displaymath} E_{k,i} := \bigcup_{I \in \cI_{k}} \{x \in I : 2^{-i}\ell(I) \leq d(x,\partial I) \leq 2^{-i + 1}\ell(I)\} =: \bigcup_{I \in \cI_{k}} B_{I,{i}}. \end{displaymath}
If $k$ is large enough, depending on $i$, the sets $B_{I,{i}}$ can be expressed as the union of intervals in $\cI_{k + 1}$. Then $B_{I,{i}}^{\leftarrow}$ makes sense, and
\begin{displaymath} \mu(E_{k,i}) = \sum_{I \in \cI_{k}} \varphi(B_{I,{i}}^{\leftarrow})\mu(I) = 2\cdot \varphi([1 - 2^{-i + 2},1 - 2^{-i + 1}]) > 0. \end{displaymath}
In probabilistic terminology, the events $E_{k,i}$ are clearly $\mu$ independent for $k \in \N$ large enough. As $\sum_{k} \mu(E_{k,i}) = \infty$ for any $i \in \N$, the Borel--Cantelli lemma yields 
\begin{displaymath}  \mu\left(\limsup_{k \to \infty} E_{k,i}\right) = 1, \qquad i \in \N, \end{displaymath} 
which is precisely what we wanted.
\end{proof}

\section{Tangent measures of $\mu$} \label{sec:6}

We want to show that $\cL^1 \ll \nu \ll \cL^1$ for every tangent measure $\nu$ of $\mu$. Using \cite[Theorem 6.9]{ref4} this is derived from the following slightly stronger statement:

\begin{proposition}
\label{tgde}
Let $\nu$ be a tangent measure of $\mu$. Then there exists a discrete set $E \subset \R$ such that
\begin{equation}\label{density} 0 < \Theta_{\ast}^{1}(\nu,z) \leq \Theta^{\ast 1}(\nu,z) < \infty\end{equation} 
for every $z \in \R \setminus E$.
\end{proposition}

We will prove Proposition \ref{tgde} by contradiction. Lemma \ref{technical} shows that if Proposition \ref{tgde} were to fail for some tangent measure $\tilde \nu$ of $\mu$, then it would also fail for a specialized tangent measure $\nu$ with certain extra features. In Lemma \ref{tgderedu} we will see that such a measure $\nu$ cannot exist.

\begin{lemma}\label{technical} Let $x \in [-1,1)$, and let $(I_{k})_{k \in \N}$ be the unique sequence of construction intervals satisfying $x \in I_{k} \in \mathcal{I}_{k}$. Suppose that Proposition \ref{tgde} fails for some measure $\tilde{\nu} \in \Tan(\mu,x)$. Then there exists a sequence $(K(i))_{i \in \N}$ of positive integers, a sequence of radii $r_i \searrow 0$, and a measure $\nu \in \Tan(\mu,x)$ with the following properties:
\begin{itemize}
\item[(i)] Proposition \ref{tgde} fails for $\nu$ inside $U(0,1)$. In other words, the points for which \eqref{density} fails for $\nu$ have an accumulation point inside $U(0,1)$. 
\item[(ii)] $$\mu(B(x,r_i))^{-1} T_{x,r_{i}\sharp}\mu \to \nu,$$
\item[(iii)] 
\begin{displaymath} \rho_{i} := \frac{\ell(I_{K(i)})}{r_{i}} \nearrow \infty \quad \text{ as } i \to \infty. \end{displaymath}
\item[(iv)]
\begin{displaymath} I_{K(i) + 1} \subset B(x,r_{i}), \qquad i \in \N. \end{displaymath} 
\item[(v)]
\begin{displaymath} B(x,5r_{i}) \subset I_{K(i)}, \qquad i \in \N. \end{displaymath} 
\end{itemize}
\end{lemma}

\begin{remark}
\label{doublingremarkk}
We will invoke the following property of tangent measures several times throughout the proof. Let $x \in [-1,1)$ and let $\nu = \lim c_i T_{x,r_i\sharp}\mu$ be any tangent measure of $\mu$ at $x$. Fix $R > 0$. Then by \cite[Theorem 1.24]{ref4} we have
$$\limsup_{i \to \infty} c_{i} \mu(B(x,Rr_{i})) \leq \nu(B(0,R)) < \infty.$$
This implies that there exists a finite constant $C_R \geq 1$ such that
\begin{align}
\label{doubling}
\sup_{i \in \N} c_i \mu(B(x,Rr_i)) \leq C_R.
\end{align}
The constant $C_{R}$ depends on $\nu$, of course, but we suppress this from the notation. 
\end{remark}

\begin{proof}[Proof of Lemma \ref{technical}] Since $\tilde{\nu}$ is a tangent measure of $\mu$, every measure of the form $cT_{0,R\sharp}\tilde{\nu}$ with $c > 0$ and $R > 0$ is also a tangent measure of $\mu$. The measure $\nu$ will be of this form. Since Proposition \ref{tgde} fails for $\tilde{\nu}$, the set of points $E$ where \eqref{density} fails has an accumulation point $z \in U(0,R_{1})$ for $R_{1} = |z| + 1$. Then, for any $c > 0$ and $R \geq R_{1}$, the set of points such that \eqref{density} fails for $cT_{0,R\sharp}\tilde{\nu}$ has an accumulation point inside $U(0,1)$. Thus, (i) is satisfied for any measure of the form 
\begin{equation}\label{2SAT} cT_{0,R\sharp}\tilde{\nu} \in \Tan(\mu,x) \end{equation} 
where $c >0$ and $R \geq R_{1}$.

Let $(\tilde{c}_{i})_{i \in \N}$ and $(\tilde{r}_{i})_{i \in \N}$ be sequences of numbers such that $\tilde{\nu} = \lim \tilde{c}_{i} T_{x,\tilde{r}_{i}\sharp}\mu$. According to \cite[Remark 14.4(1)]{ref4}, there exists $R_{2} \geq 1$ such that for any $R \geq R_{2}$, the sequence $(\tilde{c}_{i})_{i \in \N}$ can be chosen (perhaps after passage to a subsequence, which may depend on $R$) to be of the form
$$\tilde{c}_{i} = c_{R}\mu(B(x,R\tilde{r}_{i}))^{-1}$$ 
for some $c_{R} > 0$. This implies that for any $R \geq R_{2}$ we have
\begin{equation}\label{aux} \mu(B(x,R\tilde r_i))^{-1} T_{0,R\tilde r_{i}\sharp}\mu \to c_R^{-1}T_{0,R\sharp}\tilde\nu. \end{equation}
Now let $I_{k}$, $k \in \N$, be the unique construction interval satisfying $x \in I_{k} \in \cI_{k}$. For $\tilde{r}_{i} < 1$, the set of indices $k$ such that $I_{k} \not\subset B(x,\tilde{r}_{i})$ is non-empty and finite. Define $\tilde{K}(i)$ be the largest such index and write
\begin{displaymath} \tilde{\rho}_{i} := \frac{\ell(I_{\tilde{K}(i)})}{\tilde{r}_{i}} \geq 1. \end{displaymath}
Depending on the behavior of $\tilde{\rho}_i$, we now construct $K(i)$ and $r_i$ such that (i), (ii), (iii) and (iv) are simultaneously satisfied. The property (v) will essentially be a corollary of (ii), (iii) and (iv). We have two possibilities:
\begin{center} 
 (a) \quad $\sup_{i} \tilde{\rho}_{i} = \infty$ \qquad or \qquad (b) \quad $\tilde{\rho} := \sup_{i} \tilde{\rho}_{i} < \infty.$
\end{center}

\textit{Suppose that \emph{(a)} holds}. Let $R = \max\{R_1,R_2\}$, $K(i) := \tilde{K}(i)$, $r_i = R\tilde r_i$ and set $\nu = c_R^{-1}T_{0,R\sharp}\tilde \nu$. The measure $\nu$ is of the form indicated in \eqref{2SAT} and $R \geq R_1$, so it satisfies  condition (i). Since $R \geq R_2$ the property \eqref{aux} holds, so (after passage to a subsequence)
\begin{displaymath} \mu(B(x,r_{i}))^{-1}T_{x,r_{i}\sharp}\mu \to c_{R}^{-1}T_{0,R\sharp}\tilde{\nu} = \nu. \end{displaymath}
Hence $\nu$ satisfies (ii). Since $\sup \rho_i = \infty$, one more passage to a subsequence gives (iii). Finally (iv) holds, since $I_{K(i) + 1} = I_{\tilde{K}(i) + 1} \subset B(x,\tilde{r}_{i}) \subset B(x,r_{i})$.

\textit{Suppose that \emph{(b)} holds}. Set $R = \max\{R_{1},R_{2},\tilde{\rho}\}$, $K(i) := \tilde{K}(i) - 1$, $r_{i} := R\tilde{r}_{i}$ and define $\nu = c_R^{-1}T_{0,R\sharp}\tilde \nu$. Properties (i) and (ii) hold by the same proofs as in case (a). 
Moreover,
$$\rho_{i} = \frac{\ell(I_{K(i)})}{r_{i}} = \frac{\ell(I_{\tilde K(i)-1})}{R \tilde r_{i}} = 2^{\tilde K(i)}\cdot \frac{\ell(I_{\tilde K(i)})}{R \tilde r_{i}} \geq 2^{\tilde K(i)}/R \longrightarrow \infty$$
as $i \to \infty$, so that passing to a subsequence gives (iii). Finally (iv) follows from the inequalities
\begin{displaymath} \ell(I_{K(i) + 1}) = \ell(I_{\tilde{K}(i)}) \leq \tilde{\rho}\tilde{r}_{i} \leq r_{i}. \end{displaymath}

Having now fixed all our parameters, we complete the proof by showing that (v) holds. Suppose that for some $i \in \N$ we have 
\begin{align}
\label{inclusion}
B(x,5r_{i}) \not\subset I_{K(i)} =: I_{K} \quad \text{and} \quad \rho_i = \frac{\ell(I_{K(i)})}{r_{i}} \geq 100.
\end{align}
Then $B(x,5r_{i})$ intersects exactly one interval in $\mathcal{I}_{K(i)}$ besides $I_{K}$, say $I_{K}' = I_{K} + \ell(I_{K})$ (the other option being $I_{K} - \ell(I_{K})$). The common boundary point $b$ of $I_{K}$ and $I_{K}'$, see Figure 4, lies in $B(x,5r_{i})$, whence
$$B(x,r_{i}) \subset B(b,6r_{i}) \quad \mbox{and} \quad B(x,59r_{i}) \supset B(b,54r_{i}),$$
which, applying \eqref{doubling} with $R = 59$ and $\nu = \lim \mu(B(x,r_{i}))^{-1}T_{x,r_{i}\sharp}\mu$, yields
$$C_{59} \geq \frac{\mu(B(x,59r_{i}))}{\mu(B(x,r_{i}))} \geq \frac{\mu(B(b,54r_{i}))}{\mu(B(b,6r_{i}))}.$$

\begin{figure}[h]
\label{fig4}
\begin{center}
\includegraphics[scale = 0.5]{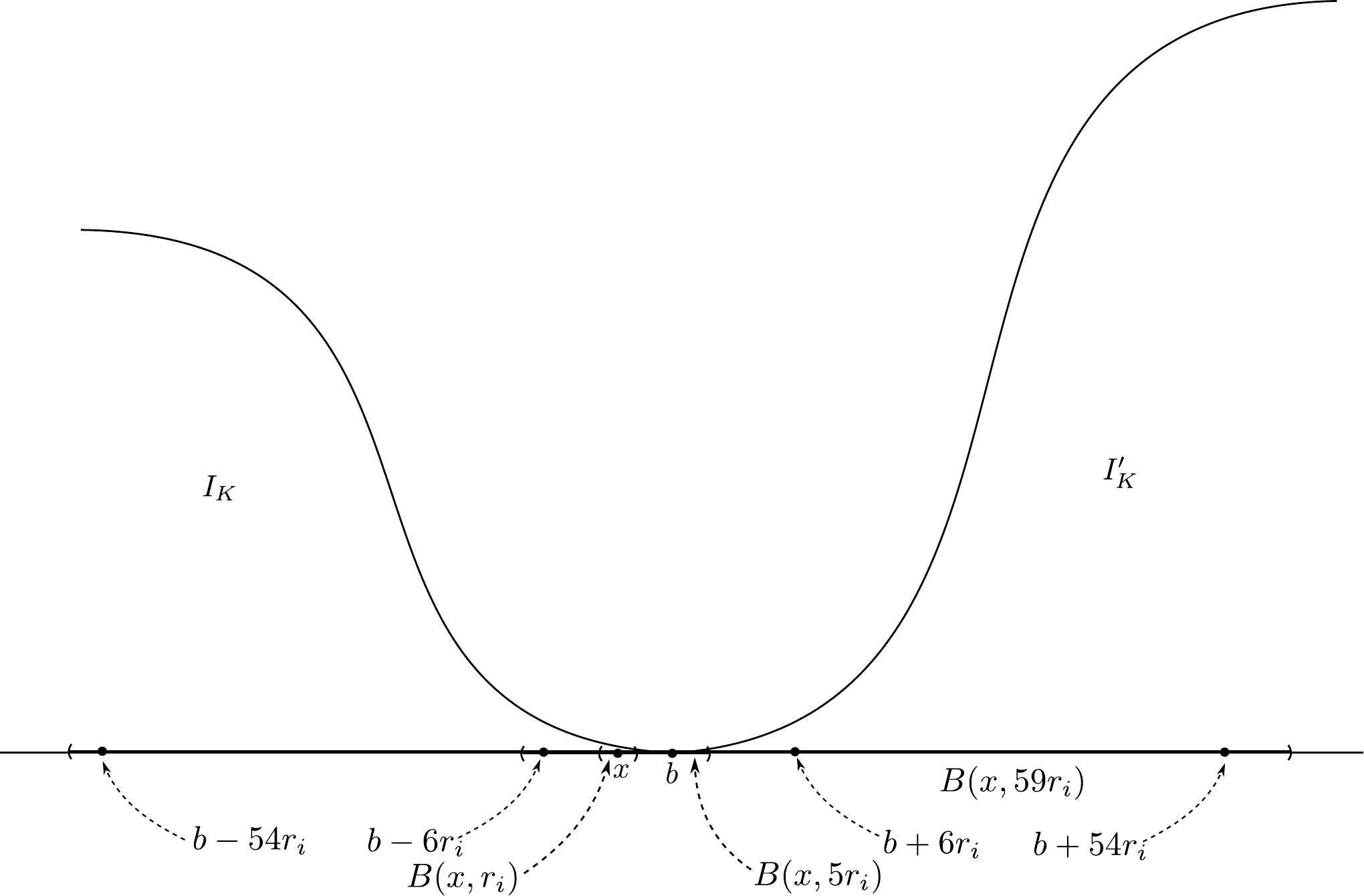}
\end{center}
\caption{The distribution of $\mu$ near the common boundary point $b$ of $I_K$ and $I_K'$.}
\end{figure}

By symmetry, we may assume that $\mu([b,b+6r_{i}]) \geq \mu([b-6r_{i},b])$. Write $J = [b,b + 6r_{i}]$ and $J^\star = [b,b + 54r_{i}]$. Then the assumptions of Lemma \ref{muprop}(3) are satisfied. In particular, $\mathcal{P}_{J} \neq \emptyset$, since $I_{K(i) + 1} \subset B(x,r_{i})$, which means that the length of the generation $K(i) + 1$ intervals is at most $2r_{i}$, and so at least one such interval is contained in $[b,b + 6r_{i}]$. Hence
$$\frac{\mu(B(b,54r_{i}))}{\mu(B(b,6r_{i}))} \geq \frac{\mu([b,b+54r_{i}])}{2\mu([b,b+6r_{i}])} \geq \tfrac{1}{2}G_{9/8, 4\lambda(i)},$$
where $\lambda(i) := \ell(J)/\ell(I_{K}) = 6r_{i}/\ell(I_K) = 6/\rho_{i}$. Hence $G_{9/8,4\lambda(i)} \leq 2C_{59}$. As $\lambda(i) \searrow 0$, Lemma \ref{phiprop}(3) now shows that \eqref{inclusion} can only hold for finitely many $i \in \N$. This, after one last passage to a subsequence, finishes the proof of (v) and the whole lemma.
\end{proof}

The following Lemma shows that the measure $\nu$ as in Lemma \ref{technical} cannot exist, and thus proves Proposition \ref{tgde}.

\begin{lemma} 
\label{tgderedu}
Suppose that $(K(i))_{i \in \N}$, $(r_i)_{i \in \N}$ and $\nu \in \Tan(\mu,x)$ satisfy the properties \emph{(ii)}, \emph{(iii)}, \emph{(iv)} and \emph{(v)} of Lemma \ref{technical}. Then there exists a finite set $E \subset U(0,1)$ such that \eqref{density} is satisfied for every $z \in U(0,1) \setminus E$. In other words, $\nu$ does not satisfy property \emph{(i)} of Lemma \ref{technical}.
\end{lemma}

The rest of the text is devoted to proving Lemma \ref{tgderedu}.

\begin{notation}Let $r_i \searrow 0$ and $x \in I_0$. If $i \in \N$, we write 
$$J^{i} := B(x,r_i),$$
and if $z \in U(0,1)$ and $\delta > 0$, we write 
$$J^{i}_\delta := B(y_i,\delta r_i),$$ 
where $y_i := x+r_i z$. Note that $J^i_\delta \subset J^i$ when $\delta < 1-|z|$.
\end{notation}

A sufficient condition for a tangent measure $\nu = \lim \mu(B(x,r_{i}))^{-1}T_{x,r_{i}\sharp}\mu$ to satisfy (\ref{density}) at a point $z \in \R$ is the following 
\begin{condition}\label{suffdensity} There exist constants $c_z > 0$ and $\delta_z > 0$ with the following property: for any $0 < \delta < \delta_z$ there exist infinitely many indices $i$ such that $(J^i_\delta,J^i)$ is $(c_z,\delta)$-comparable (recall Definition \ref{compa}). \end{condition}

Indeed, we have the inequalities
$$\limsup_{i \to \infty} \frac{T_{x,r_i\sharp} \mu(B)}{\mu(B(x,r_{i}))} \leq \nu(B) \quad \text{and} \quad \nu(U) \leq \liminf_{i \to \infty}  \frac{T_{x,r_i\sharp} \mu(U)}{\mu(B(x,r_{i}))}$$
for any compact set $B \subset \R^d$ and open set $U \subset \R^{d}$, see \cite[Theorem 1.24]{ref4}. Hence, for every $0 < \delta < \delta_z/2$, the $(c_z,\delta)$-comparability of the pairs $(B(z,\delta r_i),B(x,r_i))$ and the $(c_z,2\delta)$-comparability of the pairs $(B(z,2\delta r_i),B(x,r_i))$ for infinitely many $i$ gives
 $$c_z^{-1} \delta \leq \nu(B(z,\delta)) \leq \nu(U(z,2\delta)) \leq 2c_z \delta.$$ 
This implies
$$1/(2c_z) \leq \liminf_{\delta \searrow 0} \frac{\nu(B(z,\delta))}{2\delta} \leq \limsup_{\delta \searrow 0} \frac{\nu(B(z,\delta))}{2\delta} \leq c_z,$$
which is (\ref{density}). We will now prove Lemma \ref{tgderedu} by constructing a finite set $E \subset U(0,1)$ such that Condition \ref{suffdensity} is met for every point $z \in U(0,1) \setminus E$. We  split the proof of this into two cases -- Lemma \ref{part1} and Lemma \ref{part2} -- depending on the boundedness of the values
$$N_i := \card\cF_{J^{i}}, \quad i \in \N,$$
where $J^{i} = B(x,r_{i})$ and $\cF_{J^{i}}$ is defined using intervals in $\cI_{K(i) + 1}$, recall Definition \ref{coverpack}.

\begin{lemma}
\label{part1}
Suppose that $(K(i))_{i \in \N}$, $(r_i)_{i \in \N}$ and $\nu \in \Tan(\mu,x)$ satisfy the properties \emph{(ii)}, \emph{(iii)}, \emph{(iv)} and \emph{(v)} of Lemma \ref{technical}, and $\sup_{i} N_i = N_0 \in \N$. Then there exists a finite set $E \subset U(0,1)$ such that Condition \ref{suffdensity} is met for every $z \in U(0,1) \setminus E$.
\end{lemma}

\begin{proof}
Let $E_{i}$, $i \in \N$, be the set of all end-points of generation $K(i)+1$ construction intervals that are contained in $J^i$. In other words,
$$E_{i} := \bigcup_{I \in \cP_{J^{i}}} \partial I,$$
where the packing $\cP_{J^i}$ is defined using intervals in $\cI_{K(i)+1}$. Since $\card \cF_{J^i} \leq N_{0}$, we have $\card E_{i} \leq N_{0} + 1$. The required set $E \subset U(0,1)$ will be formed by the points $z \in U(0,1)$ such that $y_{i}$ visits infinitely often an arbitrarily small neighbourhood of the set $E_{i}$. 

We will now show that either $\liminf_i d(y_{i},E_{i})/r_{i} > 0$ for all $z \in U(0,1)$ or there exists a subsequence of indices $(i_{j})_{j \in \N}$ such that the set
\begin{equation} \label{setE} E := \Big\{z \in U(0,1) : \liminf_{j \to \infty} d(y_{i_{j}},E_{i_{j}})/r_{i_{j}} = 0\Big\}\end{equation}
is finite. Assume for the moment that there exists at least one point $z \in U(0,1)$ such that $\liminf_{i} d(y_{i},E_{i})/r_{i} = 0$, where recall that $y_i = x+r_i z$. Then we may extract a subsequence of indices $({i_{j}})_{j \in \N}$ such that $\lim_j d(y_{i_{j}},E_{i_{j}})/r_{i_{j}} = 0$. Let $E$ be the set constructed using this subsequence. We now need to show that $E$ is finite. Since $\lim_j d(y_{i_{j}},E_{i_{j}})/r_{i_{j}} = 0$, for every $j \in \N$ we may find $e_{i_{j}} \in E_{i_{j}}$ such that 
$$\lim_{j\to\infty} d(y_{i_{j}},e_{i_{j}})/r_{i_{j}} = 0.$$ 
Next suppose that $\tilde{z} \in E \setminus \{z\}$ and write $\tilde{y}_{i_{j}} = x + r_{i_{j}}\tilde{z}$. If $\liminf_j d(\tilde{y}_{i_{j}},e_{i_{j}})/r_{i_{j}} = 0$, the inequality
\begin{displaymath} d(z,\tilde{z}) = \frac{d(y_{i_{j}},\tilde{y}_{i_{j}})}{r_{i_{j}}} \leq \frac{d(y_{i_{j}},e_{i_{j}}) + d(\tilde{y}_{i_{j}},e_{i_{j}})}{r_{i_{j}}}, \qquad j \in \N,\end{displaymath}
would immediately force $\tilde{z} = z$, which shows that $\liminf_j d(\tilde{y}_{i_{j}},e_{i_{j}})/r_{i_{j}} > 0$. On the other hand, we have $\liminf_j d(\tilde{y}_{i_{j}},E_{i_{j}})/r_{i_{j}} = 0$ by definition of $\tilde{z} \in E$, which ensures that for infinitely many $j$ there exist points $\tilde{e}_{i_{j}} \in E_{i_{j}}$ such that $d(\tilde{y}_{i_{j}},\tilde{e}_{i_{j}})/r_{i_{j}} \to 0$. Then $\tilde e_{i_{j}} \neq e_{i_{j}}$ provided that $j$ is large enough. Otherwise $\tilde{z} = z$, as noted above. 

Now $N_{0}$ steps in: since $\card E_{i} \leq N_{0} + 1$ and $E_{i}$ is formed of boundary points of all generation $K(i)+1$ construction intervals contained in $J^{i} = B(x,r_{i})$, we have $d(e_{i},\tilde{e}_{i}) \geq cr_i/N_0$ for any distinct $e_{i},\tilde{e}_{i} \in E_{i}$, where $c > 0$ is some absolute constant. The points $e_{i_{j}}$ and $\tilde{e}_{i_{j}}$ must be distinct for all large enough $j$, so the inequality
\begin{displaymath} d(z,\tilde{z}) = \frac{d(y_{i_{j}},\tilde{y}_{i_{j}})}{r_{i_{j}}} \geq \frac{d(e_{i_{j}},\tilde{e}_{i_{j}}) - d(y_{i_{j}},e_{i_{j}}) - d(\tilde{y}_{i_{j}},\tilde{e}_{i_{j}})}{r_{i_{j}}}, \qquad j \in \N, \end{displaymath}
then shows that $d(z,\tilde{z}) \geq c/N_{0}$. This proves that the set $E$ is finite. 

Now, if $\liminf_{i} d(y_{i},E_{i})/r_{i} > 0$ for all $z \in U(0,1)$, we set $E = \emptyset$. Otherwise, we construct $E$ as above, in \eqref{setE}. Even in the latter case, we simplify the notation by writing $(i_{j})_{j \in \N} = (i)_{i \in \N}$, which means that $\liminf_{i} d(y_{i},E_{i})/r_{i} > 0$ for all $z \in U(0,1) \setminus E$.

Fix $z \in U(0,1) \setminus E$. Since $\liminf d(y_{i},E_{i})/r_{i} > 0$, there exists $0 < \delta_z < 1-|z|$ and $i_z \in \N$ such that $d(y_{i},E_{i}) \geq 2\delta_z r_{i}$ for all $i \geq i_z$. As $\ell(I_{K(i)})/r_{i} = \rho_{i} \nearrow \infty$ by Lemma \ref{technical}(iii), we have
$$\ell(J^{i}) = 2r_{i} < \min\{l_{3,6^8C_{5}^8}/2,1/20\} \ell(I_{K(i)})$$ 
for $i \geq i_z$, choosing a larger $i_{z}$ if necessary (at this point, recall Lemma \ref{muprop}(2)). Here $C_5$ is the constant from \eqref{doubling} with $R = 5$. Since $\nu$ satisfies Lemma \ref{technical}(i), inequality \eqref{doubling} shows that 
$$\mu(5J^{i}) \leq C_5 \mu(J^{i}).$$ 
Moreover, combining (iv) and (v) of Lemma \ref{technical} yields 
$$I_{K(i)+1} \subset J^{i} \subset 5J^{i} \subset I_{K(i)}.$$ 
Hence Lemma \ref{muprop}(2) shows that all pairs of intervals in $\cF_{J^{i}}$, $i \geq i_z$, are $(D,1)$-comparable with $D = \max\{C(1/4)^{2},6^{25}C_{5}^{25}\}$. Take $0 < \delta < \delta_z$. If $i \geq i_z$, then, since $d(y_{i},E_{i}) \geq 2\delta_z r_{i}$, we have $J^{i}_{\delta} \subset I^{i}$, for some $I^{i} \in \cI_{K(i) + 1}$. For this $I^{i}$, we then have 
$$d(J^{i}_\delta,\partial I^{i}) > \delta_z\ell(I^i),$$
see Figure 5. 

\begin{figure}[h]
\label{fig5}
\begin{center}
\includegraphics[scale = 0.5]{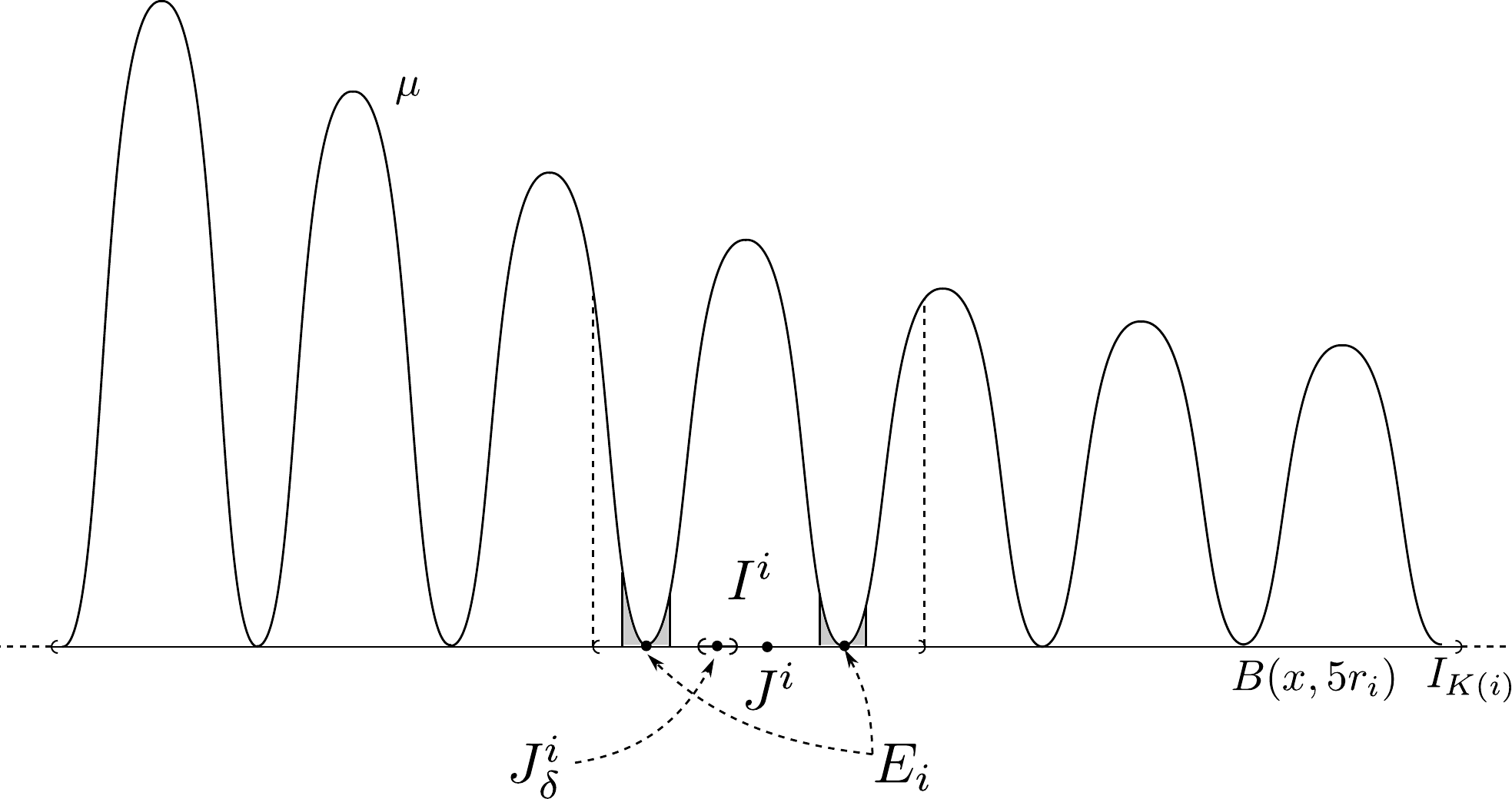}
\end{center}
\caption{The distribution of $\mu$ in the vicinity of the ball $J^{i}$. The interval $J^{i}_{\delta} \subset I^i$ cannot intersect the gray zone by the assumptions $d(y_i,E_i) \geq 2\delta_z r_i$ and $\delta < \delta_z$.}
\end{figure}

Choose $i_\delta \geq i_z$ such that we have 
$$2^{-K(i)} \leq \delta \quad \text{and} \quad K(i)+1 \geq k(\delta_z)$$ 
for all $i \geq i_\delta$ (recall Lemma \ref{muprop}(1) for the definition of $k(\delta_{z})$). Then, as $\ell(I^{i}) \leq \ell(J^i) = 2r_i$ (since $J^{i}$ contains an interval with the same generation $K(i) + 1$ as $I^{i}$), we obtain
$$\ell(J_\delta^{i}) = 2\delta r_{i} \geq \delta \ell(I^{i}) \geq 2^{-K(i)} \ell(I^{i}).$$ 
This means that for $i \geq i_\delta$ we have $\cP_{J^i_\delta} \neq \emptyset$, where the packing $\cP_{J^{i}_{\delta}}$ is now defined using intervals in $\cI_{K(i) + 2}$. Hence $J^{i}_{\delta}$ satisfies the assumptions of Lemma \ref{muprop}(1) with $k = K(i) + 1$, $J = J^i_\delta$, $I_{k} = I^{i}$ and $\tau = \delta_z$, which shows that $(J^{i}_\delta,I^{i})$ is $(C(\delta_z),\lambda)$-comparable with $\lambda = \ell(J^{i}_\delta)/\ell(I^{i})$. By definition of $N_{0}$, we have $\ell(J^{i}) \leq N_0 \ell(I^{i})$, whence
$$\delta = \frac{\ell(J^{i}_{\delta})}{\ell(J^{i})} \leq \lambda \leq N_0 \delta.$$ 
This shows the pair $(J^{i}_\delta,I^{i})$ is $(C(\delta_z)N_0,\delta)$-comparable. To finish the proof, let $I^{i}_{s}$ and $I^{i}_{l}$ be the intervals in $\cF_{J^{i}}$ with the smallest and largest $\mu$ measure, respectively. For $i \geq i_\delta$ all pairs of intervals in $\cF_{J^{i}}$ are $(D,1)$-comparable, so the inequalities
$$D^{-1} \mu(I^{i}) \leq \mu(I^{i}_{s}) \leq \mu(J^{i}) \leq N_{0}\mu(I^{i}_{l}) \leq DN_0\mu(I^{i}),$$
prove that $(I^{i},J^{i})$ is $(DN_{0},1)$-comparable. Hence the pair $(J_{\delta}^{i},J^{i})$ is $(c_z,\delta)$-comparable for $i \geq i_\delta$, where $c_z = C(\delta_z)DN_{0}^{2}$. 

This shows that Condition \ref{suffdensity} holds for $z \in U(0,1) \setminus E$.
\end{proof}

\begin{lemma}
\label{part2}
Suppose that $(K(i))_{i \in \N}$, $(r_i)_{i \in \N}$ and $\nu \in \Tan(\mu,x)$ satisfy the properties \emph{(ii)}, \emph{(iii)}, \emph{(iv)} and \emph{(v)} of Lemma \ref{technical}, and $\sup N_i =\infty$. Then Condition \ref{suffdensity} is met for every $z \in U(0,1)$.
\end{lemma}

\begin{proof}Passing to a subsequence, we may assume that $N_{i} \nearrow \infty$. Fix $z \in U(0,1)$ and let $\delta_z = 1-|z|$. Fix $0 < \delta < \delta_z$. Choose $i_\delta \in \N$ so large that the packing $\cP_{J_\delta^i} \neq \emptyset$ for all $i \geq i_\delta$. Then for these $i$ we can estimate
$$\delta N_{i}/3 \leq \card \cP_{J^{i}_{\delta}} \leq \card\cF_{J^{i}_{\delta}} \leq 3\delta N_{i},$$
where the packing $\cP_{J^{i}_{\delta}}$ and the cover $\cF_{J^{i}_{\delta}}$ are here defined using intervals in $\cI_{K(i) + 1}$. Since $\rho_i \nearrow \infty$ by Lemma \ref{technical}(iii), we can make $i_\delta$ larger if necessary, to have
$$\ell(J^{i})/\ell(I_{K(i)}) = 1/\rho_{i} < \min\Big\{l_{3,6^8C_5^8}/2,1/20\Big\}$$
for $i \geq i_\delta$. As in the proof of Lemma \ref{part1}, if we invoke \eqref{doubling} and the conditions (ii), (iii), (iv) and (v) of Lemma \ref{technical}, this inequality again yields $(D,1)$-comparability for all pairs of intervals in $\cF_{J^{i}}$, as soon as $i \geq i_\delta$. In particular, this holds for the intervals in $\cF_{J^{i}}$ with the smallest and largest $\mu$ measure, denoted $I^{i}_{s}$ and $I^{i}_{l}$. The inequalities
$$\delta N_{i} \mu(I^{i}_{s})/3 \leq \mu(J^{i}_\delta) \leq 3\delta N_{i} \mu(I^{i}_{l})$$
and
$$N_{i} \mu(I^{i}_{s})/3 \leq \mu(J^{i}) \leq N_{i} \mu(I^{i}_{l}),$$
valid for $i \geq i_\delta$, now prove that the pair $(J_{\delta}^{i},J^{i})$ is $(c_z,\delta)$-comparable for $i \geq i_\delta$, where $c_z = 9D$. 

This shows that Condition \ref{suffdensity} holds for $z \in U(0,1)$.
\end{proof}

The proof of Theorem \ref{main} is complete. 

\section{Acknowledgments}

The authors are grateful to P. Mattila and V. Suomala for valuable comments. We would also like to thank an anonymous referee for making useful suggestions for improvement throughout the manuscript.

\end{document}